\documentclass[reqno]{amsart}
\usepackage{amssymb, amstext, amscd, amsmath,verbatim,lineno, mathtools}
\usepackage{enumitem}
\usepackage[all,cmtip]{xy}
\usepackage{epigraph}
\usepackage{etoolbox}
\usepackage{comment}
\usepackage{latexsym}
\usepackage{amsthm}
\usepackage{geometry}
\usepackage{color}
\usepackage[all]{xy}
\usepackage{xcolor}

\usepackage{dsfont}

\usepackage[backend=biber,style=alphabetic, giveninits, maxalphanames=99, maxbibnames=99, maxcitenames=99,  sorting=anyt]{biblatex} 
\AtEveryBibitem{%
  \clearfield{issn} % Remove issn
  \clearfield{doi} % Remove doi

  \ifentrytype{online}{}{% Remove url except for @online
    \clearfield{url}
    }
    }
\usepackage{hyperref}
\hypersetup{colorlinks=true, linkcolor={blue}, citecolor={blue}, urlcolor={black}}
\def\V{\mathcal{V}}
\def\til{\tilde{1}}

\DeclareMathOperator{\ls}{span}

\def\source{\mathbf{s}}
\def\range{\mathbf{r}}

\def\1{1}

\usepackage{aliascnt}
\usepackage[capitalize, noabbrev]{cleveref}
\newtheorem{thm}{Theorem}[section]
\crefname{thm}{Theorem}{Theorems}

\newaliascnt{lem}{thm}
\newtheorem{lem}[lem]{Lemma}
\aliascntresetthe{lem}
\crefname{lem}{Lemma}{Lemmas}

\newaliascnt{prop}{thm}
\newtheorem{prop}[prop]{Proposition}
\aliascntresetthe{prop}
\crefname{prop}{Proposition}{Propositions}

\newaliascnt{cor}{thm}
\newtheorem{corollary}[cor]{Corollary}
\aliascntresetthe{cor}
\crefname{cor}{Corollary}{Corollary}

\newaliascnt{defn}{thm}
\theoremstyle{definition}
\newtheorem{defn}[defn]{Definition}
\aliascntresetthe{defn}
\crefname{defn}{Definition}{Definitions}

\newaliascnt{rmk}{thm}
\theoremstyle{remark}
\newtheorem{rmk}[rmk]{Remark}
\aliascntresetthe{rmk}
\crefname{rmk}{Remark}{Remarks}

\newaliascnt{example}{thm}
\newtheorem{example}[example]{Example}
\aliascntresetthe{example}
\crefname{example}{Example}{Examples}

\numberwithin{equation}{section}
\newcommand\pref[1]{(\ref{#1})} %-- without putting parentheses on labels

\usepackage{enumitem}

\def\twist{(\Sigma, i,q)}

\def\id{\epsilon}

\def\DTS{A_R(G;\Sigma)}

\def\up{^{\uparrow}}

\DeclareMathOperator{\spp}{supp}
\DeclareMathOperator{\spg}{supp_G}

\def\cob{\Sigma_{\operatorname{cob}}}

\begin{document}

\title{Cartan inverse semigroups and twisted Steinberg algebras}
\author[Brownlowe]{Nathan Brownlowe}
\author[Clark]{Lisa Orloff Clark}
\author[Lin]{Ying-Fen Lin}
\author[Naingue]{Lynnel D. Naingue}

\thanks{The first-named author was partially supported by Australian Research Council Discovery Project DP250100297. The third- and fourth-named authors would like to thank the Mathematical Research Institute (MATRIX), and to Queen's University Belfast and Mindanao State University-Iligan Institute of Technology, respectively, for providing travel support that facilitated the research collaboration for this work. They also extend their thanks to Victoria University of Wellington for the hospitality during their visits to the second-named author; the fourth-named author would like to thank also the first-named author and The University of Sydney for their hospitality during her visits.}

\address[N. Brownlowe]{School of Mathematics and Statistics, University of Sydney, NSW 2006, Australia}
\email{\href{mailto:nathan.brownlowe@sydney.edu.au}{nathan.brownlowe@sydney.edu.au}}

\address[L.O. Clark]{School of Mathematics and Statistics, Victoria University of Wellington, PO Box 600, Wellington 6140, New Zealand}
\email{\href{mailto:lisa.orloffclark@vuw.ac.nz}{lisa.orloffclark@vuw.ac.nz}}

\address[Y.-F. Lin]{Mathematical Sciences Research Centre, Queen's University Belfast, Belfast, BT7 1NN, United Kingdom}
\email{\href{mailto:y.lin@qub.ac.uk}{y.lin@qub.ac.uk}}

\address[L.D. Naingue]{Department of Mathematics and Statistics, College of Science and Mathematics, MSU-Iligan Institute of Technology, Iligan City 9200 Philippines }
\email{\href{mailto:lynnel.naingue@g.msuiit.edu.ph}{lynnel.naingue@g.msuiit.edu.ph}}

\begin{abstract}
We introduce the notion of a Cartan inverse semigroup and an (algebraic) semi-Cartan subalgebra in an $R$-algebra $A$, providing a purely algebraic analogue of Cartan semigroups and semi-Cartan subalgebras in a C*-algebra. This extends the theory of quasi-Cartan subalgebras, shifting the focus from designated subalgebras to designated subsemigroups. Our main result characterises twisted Steinberg algebras in terms of the existence of a Cartan inverse semigroup: we show that an $R$-algebra admits a Cartan inverse semigroup if and only if it is isomorphic to a twisted Steinberg algebra. Moreover, when our construction is applied to a twisted Steinberg algebra over an ample Hausdorff groupoid, the constructed twist is naturally isomorphic to the original one. 
\end{abstract}

\maketitle
\section{Introduction}

Semi-Cartan subalgebras, introduced in \cite{BCLM}, arise naturally in the study of C*-algebras. They are a generalisation of the Cartan and diagonal subalgebras studied by Renault \cite{Renault} and Kumjian \cite{Kumjian}, retaining the internal symmetry of the Cartan/diagonal setting while encompassing substantially broader classes of examples, including all group C*-algebras (full, reduced, and exotic). A central feature of this theory is the connection between algebraic structure and dynamical models, bridged by semigroups of normalisers and associated twisted groupoid constructions. 

The purpose of this paper is to develop an algebraic analogue of this theory. We show that many of the structural features of semi-Cartan subalgebras do not fundamentally depend on analytic properties, but instead arise from more basic algebraic phenomena. Passing to the algebraic setting has both advantages and challenges: while certain constructions become more flexible, one loses many of the powerful tools available in the C*-algebraic setting.     

A number of algebraic approaches to Cartan-type reconstructions have been developed, see, for example, \cite{ACCCLMRSS,BiceClark, bicejoneskud}.
Our approach is different from these and applies more generally.  Like \cite{BCLM}, we replace the role played by a distinguished subalgebra with that of a distinguished semigroup. More precisely, we consider pairs $(A,N)$, where $A$ is an $R$-algebra and $N$ is an inverse semigroup, which we call \emph{Cartan inverse semigroup}, contained in $A$ that encodes the relevant symmetry of the algebra, see \cref{defn: cis}. This framework extends the theory of quasi-Cartan subalgebras developed in \cite{ACCCLMRSS} (see \cref{prop:twoNs}) and also parallels the C*-algebraic theory of semi-Cartan subalgebras.  Importantly, $N$ is not required to be the full semigroup of normalisers, but will turn out to be a proper subsemigroup, see \cref{prop:N-in-normalisers}.

Our main result \cref{thm:reconstruct} gives an intrinsic characterisation of twisted Steinberg algebras in terms of this structure: an $R$-algebra $A$ admits a Cartan inverse semigroup if and only if $A$ is isomorphic to a twisted Steinberg algebra. Our work combines ideas from both the algebraic theory developed in \cite{ACCCLMRSS} and the C*-algebraic theory developed in \cite{BCLM}; the interaction between these perspectives is essential.  Our setup uses ideas from \cite{BCLM} and then, with everything in place, the proof of our main theorem uses techniques of \cite{ACCCLMRSS}.

In \cref{prelim}, we recall the relevant preliminaries from \cite{ACCCLMRSS} and establish a new characterisation of discrete twists in \cref{prop:twistdefTristan}. We introduce Cartan inverse semigroups and semi-Cartan subalgebras in \cref{sec:semicartan}, and investigate properties of the associated \emph{stable conditional expectations} in \cref{P}. In particular we show that the conditional expectation is unique (\cref{Punique}) and automatically faithful (\cref{Pfaithful}). In \cref{sec:main}, we prove our main reconstruction theorem (\cref{thm:reconstruct}). Finally, in \cref{sec:recovery}, we establish the recovery result (\cref{thm: recovering}), showing that applying our construction to a twisted Steinberg algebra recovers the original twist up to isomorphism.

\bigskip

\section{Preliminaries}
\label{prelim}
\subsection{Algebraic Pairs}
We recall some of the basic definitions from \cite{ACCCLMRSS}.  
Throughout this article, $R$ denotes a commutative unital ring, $R^{\times}$ is the group of units of $R$, and $A$ is an $R$-algebra.   Our main results also require $R$ to be indecomposable so that the only idempotents in $R$ are 0 and 1.  This gives us a better handle on the idempotents in our $R$-algebras. 

A commutative subalgebra $B\subseteq A$, with idempotents $I(B)$, is said to satisfy the \textit{without-torsion} condition with respect to $R$ if the following holds:
\begin{equation}
    \text{if~} e \in I(B) \setminus \{0\} \text{~and~} t\in R \text{~satisfy~} te=0, \text{~then~} t=0. \tag{WT}\label{WT}
\end{equation}
\hyperref[WT]{Condition \pref{WT}} holds automatically if $R$ is a field. The set $I(B)$ is said to form a set of \emph{local units} for $A$ if for any $\{a_1, \ldots, a_k\} \subseteq A$, there exists $e\in I(B)$ with $ea_i=a_i=a_ie$ for all $i\in \{1,\ldots, k\}$. An element $n\in A$ is a \emph{normaliser} of $B$ in $A$ if there exists $m\in A$ such that 
\begin{equation}\label{normaliser}
    mnm=m, ~ nmn=n, ~ \text{~and~} mBn \cup nBm \subseteq B,
\end{equation} and we denote the collection of normalisers of $B$ by $N(B)$. Whenever $I(B)$ forms a set of local units for $A$, the set of normalisers $N(B)$ is an inverse semigroup such that the set of idempotents in $N(B)$ equals $I(B)$ (see \cite[Lemma~2.15]{ACCCLMRSS}).

Let $A$ be an $R$-algebra, and let $B \subseteq A$ be a commutative subalgebra satisfying \hyperref[WT]{Condition \pref{WT}}.  We call a map $P:A \to B$ a \textit{conditional expectation} if $P$ is $R$-linear; $P|_{B} = \mathrm{id}_B$; and $P(bab') = bP(a)b'$ for $a \in A$ and $b,b' \in B$.
In addition, suppose:
	\begin{enumerate}[label=(\roman*), font=\normalfont]
		\item \label{paircon1}The set $I(B)$ forms a set of local units for $A$;
		\item $B=\ls_R(I(B))$;
		\item $A=\ls_R(N(B))$;
		\item \label{paircon4} There exists a \textit{faithful} conditional expectation $P:A \to B$, in the sense that for every $a \in A\setminus \{0\}$, there exists $n \in N(B)$ such that $P(na) \neq 0$. 
	\end{enumerate}
	Then the pair $(A,B)$ is called an \textit{algebraic quasi-Cartan pair} if the conditional expectation $P:A \to B$ is \textit{implemented by idempotents}, in the sense that for every $n\in N(B)$, there exists $e\in I(B)$ such that $P(n)=ne=en$.

%%%%%%%%%%%%%%%%
\subsection{Discrete twists over ample groupoids}

A \textit{topological groupoid} is a groupoid endowed with a topology such that inversion and composition are continuous.  An open set $D$ in a topological groupoid $G$ is an \textit{open bisection} if $\range(D)$ and $\source(D)$ are open, where $\range$ and $\source$ are the range and source map of the groupoid, and $\range|_D$ and $\source|_D$ are homeomorphisms onto their images. The groupoid $G$ is \textit{\'{e}tale} if $\range$ (or equivalently, $\source$) is a local homeomorphism. If $G$ is \'{e}tale and Hausdorff,  then the unit space $G^{(0)}$ is open and closed in $G$.  Also, $G$ is \'{e}tale if and only if $G$ has a basis of open bisections. We say that $G$ is \textit{ample} if it has a basis of compact open bisections. 
If $K$ and $D$ are  compact open bisections of an ample groupoid, then $K^{-1}$ and $KD$ are also compact open bisections. Under these operations, the collection of compact open bisections is an inverse semigroup \cite[Proposition~2.2.4]{Paterson}.

Given groupoids $G$ and $H$, a map $\psi: G \to H$ is called a  \emph{groupoid homomorphism} if $(\psi \times \psi)(G^{(2)}) \subseteq H^{(2)}$ and $\psi(\alpha\beta)=\psi(\alpha)\psi(\beta)$ for all $(\alpha,\beta) \in G^{(2)}$. In addition, if $\psi$ is a homeomorphism and $(\psi \times\psi)(G^{(2)}) \supseteq H^{(2)}$, then $\psi$ is called a \emph{groupoid isomorphism}. 

Next, we recall the definition of a discrete $R$-twist over an ample groupoid. We note that a discrete twist can be defined over an \'{e}tale groupoid, and one can consider a subgroup $T \leq R^{\times}$ instead of $R^{\times}$, see \cite[Definition~4.1]{TSA2022}.
A \textit{discrete $R$-twist} over an ample groupoid $G$ is a sequence 
		\[
	\xymatrix@1{
		G^{(0)} \times R^{\times} \ar[r]^-{i} ~ &~ \Sigma \ar[r]^{q}  &~ G 
	}
	\]
where the groupoid $G^{(0)} \times R^\times$ is regarded as a trivial group bundle with fibres $R^\times$, $\Sigma$ is a topological groupoid with $\Sigma^{(0)}=i(G^{(0)} \times \{1\})$, and $i$ and $q$ are continuous groupoid homomorphisms that restrict to homeomorphisms of unit spaces, such that the following conditions hold:
\begin{enumerate}
    \item[(DT1)] The sequence is \textit{exact}, in the sense that $i(\{x\} \times R^\times)= q^{-1}(x)$ for every $x \in G^{(0)}$, $i$ is injective, and $q$ is surjective. 
    \item[(DT2)] \label{DT2}The groupoid $\Sigma$ is a \textit{locally trivial $G$-bundle}, in the sense that for each $\alpha\in G$, there is an open bisection $D_\alpha$ of $G$ containing $\alpha$ and a continuous map $P_\alpha: D_\alpha \to \Sigma$ such that
    \begin{enumerate}[label=(\roman*), font=\normalfont] 
        \item \label{nD1}$q \circ P_\alpha = \operatorname{id}_{D_\alpha}$,
        \item \label{nD2}the map $(\beta, t) \mapsto i(\range(\beta), t)P_\alpha(\beta)$ is a homeomorphism from $D_\alpha \times R^\times$ to $q^{-1}(D_\alpha)$. 
    \end{enumerate}
    \item[(DT3)] The image of $i$ is \textit{central} in $\Sigma$, in the sense that $i(\range(\sigma), t)\sigma= \sigma i(\source(\sigma),t)$ for all $\sigma\in \Sigma$ and $t \in R^\times$.
\end{enumerate}

We denote a discrete $R$-twist over $G$ by $(\Sigma,i,q)$, or by $\Sigma \twoheadrightarrow G$, or simply by $\Sigma$.  
We identify $\Sigma^{(0)}$ with $G^{(0)}$ via $i$ (or via $q|_{\Sigma^{(0)}}$). Recall from \cite[Page~15]{TSA2022} that there is a free continuous action of $R^\times$ on $\Sigma$ given by 
\begin{equation}\label{eq:action}
    t \cdot \sigma= i(\range(\sigma),t)\sigma \quad \text{for all $t\in R^\times$ and $\sigma\in \Sigma$.}
\end{equation} 

Note that \cite[Definition~4.1]{TSA2022} requires $G$ and $\Sigma$ to be Hausdorff, but \cite[Corollary~2.3]{ACCCLMRSS} shows that $\Sigma$ is Hausdorff whenever $G$ is Hausdorff. Moreover, $\Sigma$ is ample whenever $G$ is ample. 
A continuous map $P_\alpha: D_\alpha \to \Sigma$ is called a (\textit{continuous}) \textit{local section} if it satisfies condition \hyperref[DT2]{(DT2)}. If $G$ is ample, then the open bisections in condition \hyperref[DT2]{(DT2)} can be chosen to be compact. We say that two discrete $R$-twists $\twist$ and $(\Sigma',i',q')$ over Hausdorff ample groupoids $G$ and $G'$, respectively, are \emph{isomorphic} if there exist groupoid isomorphisms $\psi_{\Sigma}: \Sigma \to \Sigma'$ and $\psi_G: G \to G'$ such that the following diagram commutes. 

    \begin{equation*}
        \xymatrix@1{
		G^{(0)} \times R^{\times} \ar[r]^-{i}  \ar[d]_{\psi_G|_{G^{(0)}} \times \mathrm{id}}~ &~ \Sigma \ar[r]^{q} \ar[d]^{\psi_{\Sigma}}  & G \ar[d]^{\psi_G~~} \\
		G'^{(0)} \times R^{\times}  \ar[r]^-{i'} &\Sigma'  \ar[r]^-{q'} & G'}
    \end{equation*}

\bigskip

Similar to \cite[Definition~2.1]{BCLM}, the following proposition presents an equivalent definition of a discrete twist where instead of starting with both $i$ and $q$, we start with $q$ and an action of $R^{\times}$ on $\Sigma$.  We note that condition \ref{it1:newtwist} of \cref{prop:twistdefTristan} below is not needed in the C*-algebraic setting of non-discrete twists where $R^{\times}$ is replaced by the complex unit circle $\mathbb{T}$ with the standard topology.  The compactness of $\mathbb{T}$ makes \ref{it1:newtwist} automatic, see the proof of \cite[Proposition~2.4]{BCLM}.

\begin{prop}\label{prop:twistdefTristan}
Let $\Sigma$ be a topological groupoid and let $G$ be an ample Hausdorff groupoid. Then $(\Sigma,i,q)$ is a discrete $R$-twist over $G$ if and only if the following conditions hold:
\begin{enumerate}[label=(\arabic*), font=\normalfont]
\item\label{it1:newtwist} $\Sigma^{(0)}$ is open in $\Sigma$.
\item\label{it2:newtwist} There exists a free, continuous action of $R^{\times}$ on $\Sigma$ such that for all $(\sigma, \nu) \in \Sigma^{(2)}$ and $t \in R^{\times}$,
\begin{equation}
\label{eq:actionassoc}t\cdot(\sigma \nu) = (t\cdot \sigma) \nu=\sigma(t\cdot\nu).
\end{equation}
\item\label{it3:newtwist} There exists 
a continuous, open, surjective groupoid homomorphism 
$q: \Sigma \to G$ such that $R^{\times}$ acts transitively on each fibre, that is,
$q^{-1}(q(\sigma))= R^{\times}\cdot\sigma$ for all $\sigma\in\Sigma$.
\item \label{it4:newtwist}The map $i:G^{(0)} \times R^{\times}  \to \Sigma$ is defined such that for each $t \in R^{\times}$ and $x \in G^{(0)}$,
\begin{equation}\label{eq:idef}
i(x,t) = t\cdot q|_{\Sigma^{(0)}}^{-1}(x).
\end{equation}
\end{enumerate}
\end{prop}

\begin{proof}
The forward implication can be pieced together from the literature. We have established most of items \ref{it1:newtwist} and \ref{it2:newtwist} above, with the formula for the action given by \ref{eq:action}, and it is straightforward to check that equation~\eqref{eq:actionassoc} holds. 
For item \ref{it3:newtwist}, that $q$ is open follows from \cite[Proposition~2.2(4)]{ACCCLMRSS} and that the action is transitive on each fibre follows from (DT2).  That $i$ satisfies \eqref{eq:idef} follows from \cite[Proposition~4.7]{TSA2022}.

For the reverse implication, we use \cite[Proposition~2.2(3)]{ACCCLMRSS} which says it suffices to show that $i$ and $q$ are continuous groupoid homomorphisms that restrict to homeomorphisms of unit spaces; (DT1) and (DT3) hold; the map $i$ is open; and the map $q$ is a local homeomorphism. 

For $q$, we are given that it is a continuous groupoid homomorphism.  To see that $q$ restricts to a homeomorphism of unit spaces, the argument in the proof of \cite[Proposition~2.2]{ACCCLMRSS} applies (which uses Fell's open map criterion \cite[\S II.13.2]{Fell}). 

It is straightforward to check that $i$ is a groupoid homomorphism.  Since $q|_{\Sigma^{(0)}}^{-1}$ is a well-defined continuous function and the action is continuous, we have that $i$ is continuous, and it is also straightforward to show that $i$ restricts to a homeomorphism of unit spaces.

To see that $i$ is open, it suffices to show $i(U\times \{t\})$ is open in $\Sigma$ for an open subset $U \subseteq G^{(0)}$ and $t \in R^{\times}$.  Since $q$ is a homeomorphism when restricted to the unit space and $\Sigma^{(0)}$ is open in $\Sigma$, $q|_{\Sigma^{(0)}}^{-1}(U)$ is open in $\Sigma$.  The result follows by the continuity of the action.

Finally, we show that $q$ is a local homeomorphism.  All that remains is to show that $q$ is locally injective, which we do by contradiction. Suppose that there exists $\sigma \in \Sigma$ such that for every open set $U$ containing $\sigma$, there exist distinct points $\alpha, \beta \in U$ such that $q(\alpha)=q(\beta)$. Thus $\alpha$ and $\beta$ are in the same fibre and hence there exists a $t \in R^{\times}$, $t \neq 1$, such that $\beta =t\cdot \alpha$. Thus we can find nets 
$\{t_n\} \subseteq R^{\times}\setminus \{1\}$ and $\{\alpha_n\} \subseteq \Sigma$ such that both
\[\alpha_n \to \sigma \quad \text{ and } \quad t_n\cdot \alpha_n \to \sigma \in \Sigma.\]  
Using the algebra of limits, we see that 
\[
t_n\cdot \range(\alpha_n)= t_n\cdot \alpha_n(\alpha_n^{-1})\to \sigma\sigma^{-1}= \range(\sigma).\]
Since $\Sigma^{(0)}$ is open in $\Sigma$, $t_n\cdot \range(\alpha_n) \in \Sigma^{(0)}$ eventually.  But $q$ is injective on $\Sigma^{(0)}$, so we must have $t_n\cdot \range(\alpha_n)= \range(\alpha_n)$, and hence $t_n=e$ eventually because the action is free.  This is a contradiction, thus $q$ is locally injective. 
\end{proof}

\subsection{Twisted Steinberg algebras}

Given a topological space $X$ and a ring $R$, we write $C(X,R)$ for the $R$-module of locally constant functions from $X$ to $R$. We define the \emph{support} of $f \in C(X,R)$ to be the set $\spp(f):= \{ x\in X: f(x)\neq 0 \}$.  Note that when we endow $R$ with the discrete topology, functions in $C(X,R)$ are continuous and $\spp(f)$ is a clopen set. 

Let $\twist$ be a discrete $R$-twist over an ample Hausdorff groupoid $G$. A function $f\in C(\Sigma,R)$ is said to be $R^{\times}$\textit{-contravariant} if $f(t \cdot \sigma)=t^{-1}f(\sigma)$ for all $t\in R^{\times}$ and $\sigma \in \Sigma$. 
The \emph{twisted Steinberg algebra} associated to $\twist$ over $G$ is the collection
\begin{equation}\label{TSconvolution}
    \DTS:= \{ f \in C(\Sigma,R) : f~\text{is $R^{\times}$-contravariant and $q(\spp(f))$ is compact} \},
\end{equation}
together with pointwise addition, scalar multiplication and multiplication given by the following convolution formula (see \cite[Proposition~2.9]{ACCCLMRSS}):
\begin{equation*}
    (f*g)(\sigma) = \sum\limits_{\alpha\in G^{\range(\sigma)} \cap \spg(f)} f(S(\alpha))g(S(\alpha)^{-1}\sigma),
\end{equation*}
where $f,g \in \DTS$ and $S: \spg(f) \to \spp(f)$ is any section (not necessarily continuous) for $q$ on $\spg(f):= q(\spp(f))$. 

Note that $\DTS=\ls\{\til_U: U ~\textnormal{is a compact open bisection of}~\Sigma\}$, where $\til_U$ is an element of $\DTS$ defined as
\[
\til_U(\sigma) =\begin{cases}
    t^{-1}&\text{if } \sigma \in t\cdot U \text{ for some }t \in R^{\times},\\
    0&\text{otherwise.}
\end{cases}\]
Moreover, for compact open bisections $U$ and $V$ of $\Sigma$, we have $\til_U*\til_V=\til_{UV}$.

Consider $A:= \DTS$ and $B:= A_R(G^{(0)}; q^{-1}(G^{(0)})) \subseteq A$. It follows from \cite[Lemma~4.2]{ACCCLMRSS} that the pair $(A,B)$ is an algebraic quasi-Cartan pair if and only if $\twist$ satisfies the \emph{local bisection hypothesis}, in the sense that, for every normaliser $n$ of $B$ in $A$, the set $\spg(n)$ is a bisection of $G$. Moreover, $B$ is isomorphic to the algebra $A_R(\Sigma^{(0)})$ of locally constant functions on $\Sigma^{(0)}$, and hence also isomorphic to $A_R(G^{(0)})$.  
In the sequel, we will identify them without comment. 

\section{Cartan inverse semigroups}\label{sec:semicartan}
Here we introduce our main objects of study.  
Given an inverse semigroup $N$, we denote by $n^{\dagger}$ the unique inverse of $n \in N$ satisfying $nn^{\dagger}n=n$ and $n^{\dagger}nn^{\dagger}=n^{\dagger}$. Note that the idempotents $I(N)$ of $N$ commute and $I(N)$ is equal to the set $\{n^{\dagger}n: n \in N\}$. 

\defn{\label{defn: cis} 
Let $R$ be a commutative unital ring, $A$ be an $R$-algebra, and $N\subseteq A$ an inverse semigroup. Denote by $B_N:=\ls_R(I(N))$. We say that $N$ is a \emph{Cartan inverse semigroup} of $A$ if the following conditions are satisfied:
\begin{enumerate}
    \item[(N1)]\label{cis1} $I(B_N)\subseteq N$.
    \item[(N2)]\label{cis2} $A = \ls_R (N)$.
    \item[(N3)]\label{cis3}  
    If $e\in I(N) \setminus \{0\}$ and $t\in R$ with $te = 0$, then $t =0$.
    \item[(N4)]\label{cis4}
    For any $\{n_1, \ldots, n_k\} \subseteq N$, there exists $e \in I(N)$ with $n_ie = n_i = en_i$ for all $1 \leq i \leq k$. 
    \item[(N5)]\label{cis5} If $t \in R^\times$ and $n\in N$, then $tn \in N$.
    \item[(N6)]\label{cis6} There exists a conditional expectation $P:A\to B_N$ such that for all $n\in N$,
        \begin{equation}\label{Nstable}
            P(n)n^{\dagger} \in B_N. \tag{$N$-stable}
        \end{equation}
\end{enumerate}
We call $B_N$ the \emph{(algebraic) semi-Cartan subalgebra associated to $N$.}}

\begin{rmk}
To distinguish our semigroups from C*-algebraic Cartan semigroups, we call them `Cartan inverse semigroups'; Cartan semigroups in a C*-algebra  are not, in general, inverse semigroups. 
    The word `algebraic' in front of our `semi-Cartan subalgebra' is there to avoid confusion with C*-algebraic semi-Cartan subalgebras.  In this paper, we are working solely in the algebraic setting so we omit the adjective `algebraic' in what follows.
\end{rmk}

\begin{rmk}\label{rem:BN is commutative}
The subalgebra $B_N$ is commutative: if $b = \sum_i r_i e_i$ and $b' = \sum_j s_j f_j$ with $r_i, s_j \in R$ and $e_i, f_j \in I(N)$, then $bb' = \sum_{i,j} r_i s_j e_i f_j = \sum_{i,j} s_j r_i f_j e_i = b'b$, since $R$ is commutative and the idempotents of an inverse semigroup commute.
\end{rmk}

To simplify the notation, we will often denote the semi-Cartan subalgebra associated to a Cartan inverse semigroup $N$ simply by $B$ unless we need to worry about different associated subalgebras.

Our first observation, \cref{equistat:ibi} below, is that Condition \eqref{Nstable} for a conditional expectation $P$ is equivalent to requiring that $P$ be ``implemented by idempotents" on elements of $N$.  Note that in \cite[Definition~3.3]{ACCCLMRSS} this implementation is for all normalisers. In our more general setting, we only require it to hold on the Cartan inverse semigroup $N$. To show that, first we need a lemma. 

\begin{lem}\label{lem:P-identities}
	Let $N$ be a Cartan inverse semigroup of $A$ with associated semi-Cartan
	subalgebra $B$, and let $P\colon A\to B$ be an $N$-stable conditional
	expectation. Then for every $n\in N$:
	\begin{enumerate}[label=(\arabic*), font=\normalfont]
		\item\label{it:P-a} $P(n)=nn^{\dagger}P(n)=P(n)n^{\dagger}n$;
		\item\label{it:P-b} $P(n)n^{\dagger}\in I(N)$ and $n^{\dagger}P(n)\in I(N)$;
		\item\label{it:P-c} $P(n)n^{\dagger}=P(n)P(n^{\dagger})=P(n^{\dagger})P(n)=n^{\dagger}P(n)$.
	\end{enumerate}
\end{lem}

\begin{proof}
	Throughout, fix by (N4) an $e_{0}\in I(N)$ with $ne_{0}=n=e_{0}n$ and
	$n^{\dagger}e_{0}=n^{\dagger}=e_{0}n^{\dagger}$.
	
	\ref{it:P-a} Since $nn^{\dagger},n^{\dagger}n\in I(N)\subseteq B$, the
	bimodule property gives $P(n)=P((nn^{\dagger})n)=nn^{\dagger}P(n)$ and
	$P(n)=P(n (n^{\dagger}n))=P(n)n^{\dagger}n$.
	
	\ref{it:P-b} Set $f:=P(n)n^{\dagger}$, which lies in $B$ by \eqref{Nstable}.
	By \ref{it:P-a}, $fn=P(n)n^{\dagger}n=P(n)$ and
	$fnn^{\dagger}=P(n)n^{\dagger}=f$. Since $P(n)\in B$ and $P|_{B}=\mathrm{id}_B$,
	\[
	fn=P(n)=P\bigl(P(n)\bigr)=P(fn)=fP(n)=f^{2}n,
	\]
	and multiplying on the right by $n^{\dagger}$ gives $f^{2}nn^{\dagger}=fnn^{\dagger}$,
	that is, $f^{2}=f$. Hence $f\in I(B)=I(N)$ by (N1).
	
	Now set $g:=n^{\dagger}P(n)=n^{\dagger}fn$, which is a product of elements of
	$N$ and so lies in $N$. Since $f$ and $nn^{\dagger}$ are commuting idempotents,
	$f(nn^{\dagger})f=f^{2}nn^{\dagger}=fnn^{\dagger}=f$, and therefore
	\[
	g^{2}=n^{\dagger}f(nn^{\dagger})fn=n^{\dagger}fn=g,
	\]
	so $g\in I(N)$.
	
	\ref{it:P-c} By \ref{it:P-b} we have $f,g\in I(N)\subseteq B$, so $P(f)=f$
	and $P(g)=g$. Using the bimodule property with $P(n),P(n^{\dagger})\in B$,
	together with $P(n^{\dagger})e_{0}=P(n^{\dagger}e_{0})=P(n^{\dagger})$ and
	$e_{0}P(n^{\dagger})=P(n^{\dagger})$, we obtain
	\[
	f=P(f)=P\bigl(P(n)\,n^{\dagger}e_{0}\bigr)=P(n)P(n^{\dagger})e_{0}=P(n)P(n^{\dagger}),
	\]
	\[
	g=P(g)=P\bigl(e_{0}n^{\dagger}\,P(n)\bigr)=e_{0}P(n^{\dagger})P(n)=P(n^{\dagger})P(n).
	\]
	Since $B$ is commutative (see \cref{rem:BN is commutative}), we have
	$P(n)P(n^{\dagger})=P(n^{\dagger})P(n)$, and hence $f=g$.
\end{proof}
 
\begin{prop}\label{equistat:ibi}
    Suppose that $N \subseteq A$ is an inverse semigroup satisfying conditions (N1) -- (N5)
    %\ref{cis1} -- \ref{cis5}
    of \cref{defn: cis} with $B=\ls_R(I(N))$, and that $P:A \to B$ is a conditional expectation. Then Condition \eqref{Nstable} on $P$ is equivalent to the following statement:
    \begin{enumerate}
        \item[] For all $n\in N$, there exists $e\in I(N)$ such that 
        \begin{equation}\label{ibi}
            P(n)=ne=en. 
        \end{equation}
    \end{enumerate}
\end{prop}

\begin{proof}  Suppose $P$ satisfies Condition \eqref{Nstable}.
     Fix $n\in N$. Then $P(n)= nn^{\dagger}P(n) = P(n)n^{\dagger}n$.  
    Consider $e:=n^{\dagger}P(n)$ and $e':=P(n)n^{\dagger}$. Then by 
    \cref{lem:P-identities}\ref{it:P-c}, we have $e=e' \in I(N)$, as needed.
    
    Conversely, suppose that $P$ satisfies \pref{ibi}. Then there exists $e\in I(N)$ such that $P(n)=en$. Hence, $P(n)n^{\dagger}=enn^\dagger \in I(N)=I(B) \subseteq B$, implying that $P$ satisfies Condition \eqref{Nstable}.  
\end{proof}

The motivating example is that \emph{all} twisted Steinberg algebras contain a Cartan inverse semigroup, which we show in \cref{prop:twoNs}. More precisely, given an indecomposable commutative ring $R$ and a discrete $R$-twist $\Sigma$ over an ample Hausdorff groupoid $G$, consider the following collections:
\begin{align}\label{N0}
N_0&:= \{\tilde{1}_V: V \, \text{is a compact open bisection of }  \Sigma\}, ~\textnormal{and}
\\
N_1&:= \{f\in A_R(G; \Sigma): q(\spp(f)) \subseteq V \, \text{ for some open bisection} \, V \subseteq G \ \text{ and } \, f(\Sigma) \subseteq R^{\times}\cup \{0\}\}.
\end{align}
We show that $N_0$ is a Cartan inverse semigroup and that $N_0$ and $N_1$ coincide in \cref{prop:twoNs}. We thus use $N_0$ and $N_1$ interchangeably as suits our purpose. Note that there still might be other Cartan inverse semigroups sitting inside the twisted Steinberg algebra $A:=A_R(G; \Sigma)$, that is,  $N_0$ (or $N_1$) is not necessarily the unique one to generate $A_R(G^{(0)})$. 

\begin{prop}\label{prop:twoNs}
    The collection $N_0$ is a Cartan inverse semigroup for $\DTS$ and its associated semi-Cartan subalgebra $B_{N_0}$ is isomorphic to $A_R(G^{(0)})$. Moreover, $N_0=N_1.$
\end{prop}

\begin{proof}
To see that $N_0$ is an inverse semigroup, recall from \cite[Lemma~4.2]{ACCCLMRSS} that for any compact open bisections $U,V$ of $\Sigma$, we have $\til_U\til_V=\til_{UV} \in N_0$ since $UV$ is a compact open bisection of $\Sigma$, and $\til_{V}^\dagger=\til_{V^{-1}}$ is the inverse of $\til_V$. We also have $A=\ls_R(N_0)$ and $B_{N_0} := \ls_R\{\tilde{1}_V^{\dagger} \tilde{1}_V: \tilde{1}_V\in N_0\}$ is isomorphic to $B:=A_R(G^{(0)};q^{-1}(G^{(0)}))$ which is also isomorphic to $A_R(G^{(0)})$ with a faithful conditional expectation $P: A \to B$, given by restriction of functions from $\Sigma$ to $q^{-1}(G^{(0)})$. Note that since $R$ is indecomposable, we have $I(B_{N_0})=I(N_0)\subseteq N_0$, and so (N1) holds.

Next, we claim that the conditional expectation $P$, when restricted to $N_0$, is implemented by idempotents of 
$B$, so that by \cref{equistat:ibi}, $P$ is $N_0$-stable.   
 The argument is similar to the one in \cite[Lemma~4.2]{ACCCLMRSS} (although we need to be a little careful with our identification of $A_R(G^{(0)})$ and $A_R(G^{(0)};q^{-1}(G^{(0)}))$).  Fix $\tilde{1}_V \in N_0$.    
 Let 
 \[
 U= \operatorname{supp}(\tilde{1}_V) \cap \Sigma^{(0)}
 \] and let $e= \tilde{1}_U$. Then a routine calculation shows that $P(\tilde{1}_V) = e\tilde{1}_V=\tilde{1}_Ve$, proving the claim.  Moreover, for $t\in R^\times$, $t\til_V=\til_{t\cdot V}\in N_0$, since $t\cdot V$ is also a compact open bisection of $\Sigma$, and so (N5) holds. We also see that for $t\in R$ and $\til_U \in I(N_0)\setminus \{0\}$ with $t\til_U=0$ we have $t=0$, and so (N3) holds. 

To see that $N_0$ satisfies (N4), take a finite collection $\{\til_{V_i}\}_{i=1}^k$ and let $f=\tilde{1}_W$, where 
\[
W = \bigcup_{i=1}^k(\source(V_i) \cup \range(V_i)).
\]
Then $f\in I(N_0)$ satisfies $\til_{V_i}f=\til_{V_i}=f\til_{V_i}$ for all $i=1, \ldots, k$. 

For the last statement, it is clear that $N_0 \subseteq N_1$.  Since $N_1 \subseteq N(B)$, the reverse containment follows from \cite[Proposition~8.6]{ACCCLMRSS}. 
\end{proof}

\begin{rmk}
   Let $\mathcal{B}$ be a basis of compact open bisections such that $\mathcal{B}$ is an inverse semigroup. The collection $N_2:= \{\tilde{1}_X: X\in \mathcal{B}$\} need not be a Cartan inverse semigroup.  Here we would have $\ls_R(I(N_2))\cong A_R(G^{(0)})$ but there might be additional idempotents of $\ls_R(I(N_2))$ that are not in $N_2$. 
\end{rmk}

Before showing how semi-Cartan subalgebras generalise quasi-Cartan and graded quasi-Cartan subalgebras (see \cref{prop:NB-alg-semiCartan} and \cref{gqc}), in the remainder of this section, we establish some properties of Cartan inverse semigroups and their corresponding semi-Cartan subalgebras, then in \cref{P}, we establish some properties of stable conditional expectations.  

\begin{prop}\label{res: scB} 
Let $N$ be a Cartan inverse semigroup of $A$. Then $B_N$ is a commutative subalgebra of $A$ such that 
\begin{enumerate}[label=(\alph*), font=\normalfont]
   \item\label{scB1} $B_N=\ls_R(I(B_N))$,
    \item\label{scB2} $B_N$ satisfies Condition \pref{WT}, and
    \item\label{scB3} $I(B_N)$ is a set of local units for $A$.
\end{enumerate}
\end{prop}

\begin{proof}
Commutativity of $B_N$ is addressed in \cref{rem:BN is commutative}. Part \ref{scB1} immediately follows since $I(N)=I(B_N)$. Then part \ref{scB2} follows from (N3) and part \ref{scB3} follows from (N2) and (N4).
\end{proof}

We next consider the relationship between a Cartan inverse semigroup and the semigroup of normalisers.

\begin{prop}\label{prop:N-in-normalisers}
	Let $N$ be a Cartan inverse semigroup of $A$ with associated semi-Cartan
	subalgebra $B$.
	\begin{enumerate}[label=(\arabic*), font=\normalfont]
		\item\label{it:N38-1} For every $n\in N$, we have
		$nBn^{\dagger}\cup n^{\dagger}Bn\subseteq B$; in particular
		$N\subseteq N(B)$.
		\item\label{it:N38-2} If $P\colon A\to B$ is an $N$-stable conditional
		expectation, then for every $n\in N$ we have $P(n)\in N$ with
		$P(n)^{\dagger}=P(n^{\dagger})$; in particular $P(n)\in N(B)$.
	\end{enumerate}
\end{prop}

\begin{proof}
	\ref{it:N38-1} It suffices to show that $nI(N)n^{\dagger}\subseteq I(N)$ for
	every $n\in N$. If $e\in I(N)$ then $nen^{\dagger}\in N$ since $N$ is an inverse
	semigroup, and since $n^{\dagger}n$ and $e$ are commuting idempotents,
	\[
	(nen^{\dagger})^{2}=n\,e\,(n^{\dagger}n)\,e\,n^{\dagger}
	=n\,(n^{\dagger}n)\,e^{2}\,n^{\dagger}
	=(nn^{\dagger}n)en^{\dagger}=nen^{\dagger}.
	\]
	Hence $nen^{\dagger}\in I(N)$, and a similar argument gives
	$n^{\dagger}en\in I(N)$.
	
	\ref{it:N38-2} Applying \cref{lem:P-identities}\ref{it:P-c} to $n$ and to $n^{\dagger}$,
	\[
	P(n)n^{\dagger}=P(n)P(n^{\dagger})=P(n^{\dagger})P(n)=P(n^{\dagger})(n^{\dagger})^{\dagger},
	\]
	so $n$ and $n^{\dagger}$ determine the same idempotent
	\[
	e:=P(n)n^{\dagger}=n^{\dagger}P(n)=P(n^{\dagger})n=nP(n^{\dagger})\in I(N).
	\]
	By \cref{lem:P-identities}\ref{it:P-a} we have
	$en=P(n)n^{\dagger}n=P(n)$ and $ne=nn^{\dagger}P(n)=P(n)$; applying
	\cref{lem:P-identities}\ref{it:P-a} to $n^{\dagger}$ gives
	$en^{\dagger}=P(n^{\dagger})nn^{\dagger}=P(n^{\dagger})$ and
	$n^{\dagger}e=n^{\dagger}nP(n^{\dagger})=P(n^{\dagger})$. In particular
	$P(n)=en\in N$, being a product of elements of $N$. Hence
	\[
	P(n)^{\dagger}=(en)^{\dagger}=n^{\dagger}e^{\dagger}=n^{\dagger}e=P(n^{\dagger}),
	\]
	and $P(n)\in N\subseteq N(B)$ by \ref{it:N38-1}.
\end{proof}

Recall that there is a natural partial order on any inverse semigroup given by $s\leq t$ if and only if $s= ss^\dagger t$ (or, equivalently, if and only if $s= ts^\dagger s$). The partial order is preserved by
multiplication and inversion, and the collection of idempotents is a meet semilattice such that the product of two idempotents is their meet. Furthermore, $s\leq t$ if and only if there is an idempotent $e$ such that $s= te$, and this is equivalent to the existence of an idempotent $f$ such that $s= ft$. 

Given an $R$-algebra $A$, we know by \cref{prop:N-in-normalisers}\ref{it:N38-1} that any Cartan inverse semigroup $N$ satisfies $N\subseteq N(B_N)$; thus, Lemma~2.17 in \cite{ACCCLMRSS} holds for Cartan inverse semigroups. As we will use this lemma for later proofs, for self-containment we include it as a corollary and adapt it to our setting below. 

\begin{corollary}\label{lem:orders}
Suppose that $A$ is an $R$-algebra, and $N$ is a Cartan inverse semigroup. Then the following hold.
\begin{enumerate}[label=(\arabic*), font=\normalfont]
    \item If $n\in N$ and $e\in I(B_N)$ satisfy $n\leq e$, then $n\in I(B_N)$.
    \item For $f, e\in I(B_N)$ and $t\in R^\times$, if $0\neq f\leq te$, then $t= 1$.
    \item For $n, m\in N$, we have $nn^\dagger m= n \Longleftrightarrow mn^\dagger n= n$.
    \item For all $n\in N$ and $e\in I(B_N)$, we have $en, ne\leq n$.
    \item If $e\in I(B_N)$ and $n\in N$, then $n- ne, n- en\in N$, with $(n- ne)^\dagger = n^\dagger - en^\dagger$ and $(n- en)^\dagger = n^\dagger - n^\dagger e$, and $n- ne, n- en\leq n$.
\end{enumerate}
\end{corollary}

\section{Properties of a stable conditional expectation} \label{P}
We have two main results in this section. The first, \cref{Punique}, says that an $N$-stable conditional expectation is unique, and the second, \cref{Pfaithful} says that it is also automatically faithful. With these, we demonstrate how semi-Cartan subalgebras generalise quasi-Cartan and graded quasi-Cartan subalgebras in \cref{prop:NB-alg-semiCartan} and \cref{gqc}.

Define a relation $\sqsubseteq$ on the Cartan inverse semigroup $N$ by
\[m \sqsubseteq n \iff \textnormal{there exists }b\in B \textnormal{ such that } m=mb=nb. \] 

\begin{prop}\label{Punique}
    For all $n\in N$, $P(n) \sqsubseteq n$.
    In particular, 
    \begin{equation}\label{maxP}
        P(n)= \max \{b \in B \cap N: b \sqsubseteq n \} \quad\text{(where the max is with respect to $\sqsubseteq$)}.
    \end{equation}
Consequently, the $N$-stable conditional expectation $P$ is unique.
\end{prop}

\begin{proof}
    It follows from \cref{equistat:ibi} that $P(n) \in N$. Take $b:= n^{\dagger}P(n) \in B$. It is straightforward to check that $P(n)=P(n)b=nb$ implying that $P(n)\sqsubseteq n$. 

    Now, suppose that $b \sqsubseteq n$ for some $b\in B\cap N$. Then there exists $b' \in B$ such that $b= bb'=nb'$. Take $b'':= b^{\dagger}bb'$. Then $bb'' = bb' = b$ and 
    \[P(n)b'' = P(n) b'b^{\dagger}b = P(nb')b^{\dagger}b = P(b)b^{\dagger}b = b. \] 
    Hence, $b \sqsubseteq P(n)$ so that $P(n)$ is the maximum $b \in B \cap N$ such that $b \sqsubseteq n$. The uniqueness of the $N$-stable conditional expectation $P$ follows directly from \eqref{maxP}.
\end{proof}

\begin{lem}\label{lem:n-P(n)}
	For each $n\in N$, we have $n-P(n)\in N$.
\end{lem}

\begin{proof}
	We have $n-P(n)=n-P(n)n^\dagger n=(nn^\dagger  -P(n)n^\dagger)n$, and so it suffices to prove that $nn^\dagger  -P(n)n^\dagger\in N$. We claim that $nn^\dagger  -P(n)n^\dagger\in I(B)\subseteq N$. We know that $nn^\dagger  -P(n)n^\dagger\in B$, and so we only need to show that this element is an idempotent. We use $P(n)n^\dagger\in B$ to get 
	\begin{align*}
	(nn^\dagger  -P(n)n^\dagger)^2&=nn^\dagger  -P(n)n^\dagger  - P(n)n^\dagger +P(n)n^\dagger P(n)n^\dagger\\
	&= nn^\dagger  -P(n)n^\dagger  - P(n)n^\dagger +P(P(n)n^\dagger n)n^\dagger\\
	&= nn^\dagger  -P(n)n^\dagger  - P(n)n^\dagger +P(P(n))n^\dagger\\
	&= nn^\dagger  -P(n)n^\dagger  - P(n)n^\dagger +P(n)n^\dagger\\
	&= nn^\dagger  -P(n)n^\dagger.
	\end{align*}
	So the claim holds, and the result follows.
\end{proof}

Unlike in quasi-Cartan subalgebras (\cite[Definition 3.3]{ACCCLMRSS}), conditional expectations associated to a semi-Cartan subalgebra are not assumed to be faithful; nevertheless, faithfulness is automatic in the sense that
\begin{equation}\label{eq:Pfaithful}
	a\in A,\, a\not=0\implies \text{there exists } m\in N\text{ with }P(ma)\not=0.
\end{equation} 

\begin{prop}\label{Pfaithful}
    The stable conditional expectation $P$ in \cref{defn: cis} is faithful. 
\end{prop}

\begin{proof}
We adapt the argument in the proof of Theorem~11.10 of \cite{BCLM} and use an induction argument. Note that the join of two idempotents $e$ and $f$ is $e\vee f:=e+f-ef$, and we extend this to the join of finitely many idempotents in the natural way.

First suppose $a=rn\not= 0$ for $r\in R$ and $n\in N$. Then $rn^\dagger n\not =0$. For $m=n^\dagger\in N$, we have 
	\[
	P(ma)=P(n^\dagger rn)= r P(n^\dagger n)=rn^\dagger n\not =0.
	\]
	Assume that \eqref{eq:Pfaithful} holds for all nonzero elements of $N_k:=\left\{\sum_{i=1}^kr_in_i: r_i\in R, n_i\in N\right\}$, and let $a:=\sum_{i=1}^{k+1}r_in_i$ be a nonzero element of $N_{k+1}$. We claim that $n_j^\dagger a\not=0$ for some $1\le j\le k+1$. To see this, assume for contradiction that $n_j^\dagger a=0$ for all $1\le j\le k+1$. Then we have $\left(\bigvee_{j=1}^{k+1}n_jn_j^\dagger\right)a=0$, and since $\bigvee_{j=1}^{k+1}n_jn_j^\dagger$ dominates each $n_in_i^\dagger$ we get 
	\begin{align*}
	a=\sum_{i=1}^{k+1}r_in_i=\sum_{i=1}^{k+1}n_in_i^\dagger r_in_i &= \sum_{i=1}^{k+1}\left(\bigvee_{j=1}^{k+1}n_jn_j^\dagger\right)n_in_i^\dagger r_in_i\\
	&= \left(\bigvee_{j=1}^{k+1}n_jn_j^\dagger\right)\sum_{i=1}^{k+1}n_in_i^\dagger r_in_i\\
	&= \left(\bigvee_{j=1}^{k+1}n_jn_j^\dagger\right)a\\
	&= 0.
\end{align*}
Since this contradicts that $a\not=0$, the claim holds. 

By potentially renumbering we can assume that $n_{k+1}^\dagger a\not=0$. If $P(n_{k+1}^\dagger a)\not=0$, then we choose $m=n_{k+1}^\dagger$ in \eqref{eq:Pfaithful} and we are done. So assume that $P(n_{k+1}^\dagger a)=0$. Then using that $n_{k+1}^\dagger n_{k+1}-P(n_{k+1}^\dagger n_{k+1})=0$ for the last equality we get
\[
n_{k+1}^\dagger a = n_{k+1}^\dagger a-P(n_{k+1}^\dagger a)=\sum_{i=1}^{k+1}r_i(n_{k+1}^\dagger n_i-P(n_{k+1}^\dagger n_i))=\sum_{i=1}^k r_i(n_{k+1}^\dagger n_i-P(n_{k+1}^\dagger n_i)),
\]
and we know from \cref{lem:n-P(n)} that each $n_{k+1}^\dagger n_i-P(n_{k+1}^\dagger n_i)\in N$. Hence $n_{k+1}^\dagger a\in N_k$. By the inductive hypothesis we know that there is $m'\in N$ such that $P(m'n_{k+1}^\dagger a)\not=0$, and then taking $m=m'n_{k+1}^\dagger$ shows that $a$ satisfies \eqref{eq:Pfaithful}. Since $A=\operatorname{span}_R(N)=\bigcup_{k=1}^\infty N_k$, we see that $P$ is faithful. 
\end{proof}

In our next proposition, we consider the relationship between the semigroup of normalisers in quasi-Cartan subalgebras and Cartan inverse semigroups.  

\begin{prop}\label{prop:NB-alg-semiCartan}
    Let $R$ be a commutative unital ring, $A$ be an $R$-algebra, and $B\subseteq A$ a commutative subalgebra satisfying Condition (WT). Then $N:=N(B)$ forms a Cartan inverse semigroup with semi-Cartan subalgebra associated to $N$ equal to $B$ if and only if the pair $(A,B)$ satisfies all the properties of being an algebraic quasi-Cartan pair without the assumption that the conditional expectation $P:A\to B$ is faithful. 
\end{prop}

\begin{proof}
It follows from the definition of a Cartan inverse semigroup $N:=N(B)$ that $I(B)=I(N)$. Hence all the required conditions for the pair $(A,B)$ to be algebraic quasi-Cartan are immediate except for showing that the stable conditional expectation $P:A \to B$ is faithful and implemented by idempotents. However, by \cref{Pfaithful}, $P$ is faithful and by \cref{equistat:ibi}, $P$ is implemented by idempotents, as required. 

Conversely, to see that $N:=N(B)$ is a Cartan inverse semigroup, note that $I(B)=I(N(B))=I(N)$ by \cite[Lemma~2.15]{ACCCLMRSS} so that
$B_N:= \ls_R(I(N)) = \ls_R(I(B))= B$.
Now, all the required conditions are immediate except for \eqref{Nstable}. By \cite[Proposition~3.7]{ACCCLMRSS}, there is a unique conditional expectation $P$ that is implemented by idempotents. Thus, by \cref{equistat:ibi}, $P$ is $N$-stable, completing the proof.
\end{proof}

Combining \cref{Pfaithful} and \cref{prop:NB-alg-semiCartan}, the following corollary is immediate and tells us that we can remove `faithfulness' from the definition of an algebraic quasi-Cartan pair; it comes for free.

\begin{corollary}
    \label{cor:uniquePforaqcp} The pair $(A,B)$ is an algebraic quasi-Cartan pair if and only if
 $(A,B)$ satisfies all the properties of an algebraic quasi-Cartan pair without the assumption that the conditional expectation $P:A\to B$ is faithful.
\end{corollary}

%%%%%%%%%%%%%%

In our next example,  we consider the graded quasi-Cartan pairs of \cite{CNV25}.
\begin{example}\label{gqc}
Suppose that $\Gamma$ is a group with identity $\id$
and $A=\bigoplus_{\gamma \in \Gamma} A_\gamma$ is a graded $R$-algebra. Let $(A,B)$ be a graded algebraic quasi-Cartan pair (as in \cite{CNV25}), and consider $N:=N_\star(B)$, which is the set of homogeneous normalisers of $B$. Since $I(N_\star(B))=I(B)$, it follows that $B_N$ is equal to $B$, implying that $\ls(I(N_\star(B)))=\ls(I(B))=B$. All the required conditions for $N$ to be a Cartan inverse semigroup are straightforward to check except for Condition \eqref{Nstable}. But the map $P:= P' \circ \Phi_\id$ defines an $N$-stable conditional expectation, where $P':A_\id \to B$ is the conditional expectation implemented by idempotents and $\Phi_\id: A\to A_\id$ is defined by $a= \sum_{\gamma\in \Gamma} a_\gamma \mapsto a_\id$.  
\end{example}

\begin{example} 
Recall from \cite[Theorem~8.7]{ACCCLMRSS} that the subalgebra $A_R(G^{(0)})$ is a quasi-Cartan subalgebra of a twisted Steinberg algebra $A_R(G;\Sigma)$ if and only if the twist satisfies the local bisection hypothesis.  This restricts the twists under consideration in \cite{ACCCLMRSS}.  For example, there are group algebras that do not satisfy the local bisection hypothesis. In the setting of complex group algebras, finite cyclic groups with prime order have normalisers of the subalgebra $\mathbb{C}\delta_{e}$ that have support equal to the entire group, see \cite[Lemma~5.2]{ABCCLMR}.  

In such groups, the semigroup of normalisers of $\mathbb{C}\delta_{e}$ in $\mathbb{C}[G]$ is not a Cartan inverse semigroup.  There are a few different ways to see this, for example, it is a consequence of \cref{prop:NB-alg-semiCartan}.  As a reality check,  in the example below, we demonstrate directly that the conditional expectation is not $N$-stable.  

Let $R$ be some field that is not of characteristic $2$ or $3$. Let $G=C_2=\langle a \mid a^2=e\rangle$, and take $A$ to be the untwisted group ring $R[C_2]$, $B=R\delta_e$ and $P$ be the restriction of `functions' in $R[C_2]$ to $B$.  Then we know from \cite[Example~3.9]{ACCCLMRSS} that $N(B)=\{0\}\cup A^\times$.  If we let $n:=2\delta_e+\delta_a$, then $n\in N(B)$ with $n^\dagger=\frac{1}{3}(2\delta_e-\delta_a)$ because 
\[
n^\dagger n = \frac{1}{3}(2\delta_e-\delta_a)(2\delta_e+\delta_a)=\frac{1}{3}(4\delta_e-\delta_{a^2})=\delta_e=nn^\dagger.
\]
But 
\[
P(n)n^\dagger = 2\delta_e\left(\frac{1}{3}(2\delta_e-\delta_a)\right)=\frac{4}{3}\delta_e-\frac{2}{3}\delta_a\not\in B
\]
demonstrating that $P$ is not $N$-stable.
However, as a special case of \cref{prop:twoNs}, we do have the Cartan inverse semigroup
\[
N:=\{0\}\cup\{t\delta_g : t\in R^\times,\, g\in C_2\} \subseteq N(B)
\] and the restriction of $P$ to $N$ is stable.  
(Note that $N$ is just $N_0=\{\til_X : X\text{ a compact open bisection of }\Sigma\}$.)  Group algebras also fit into the context of graded quasi-Cartan pairs as in \cref{gqc} (see \cite[Example~8.1]{CNV25}).
\end{example}

\section{From Cartan inverse semigroups to twisted Steinberg algebras}\label{sec:main}

This section is devoted to proving the main theorem of the paper:

\begin{thm}\label{thm:reconstruct}
    Let $R$ be an indecomposable commutative unital ring. Let $A$ be an $R$-algebra and $B_N\subseteq A$ be the semi-Cartan subalgebra associated to a Cartan inverse semigroup $N$. Then there is a discrete $R$-twist $\Sigma \twoheadrightarrow G$ over an ample Hausdorff groupoid $G$ and an isomorphism $\varphi: A \to A_R(G;\Sigma)$ which maps $B_N$ onto the Steinberg algebra $A_R(G^{(0)})$. Furthermore, if the Cartan inverse semigroup $N$ is summable, i.e., for all $m, n\in N$, 
    \begin{equation}\label{defn: summable}
    m^{\dagger}n , mn^{\dagger} \in B_N \implies m+n \in N, \tag{Summable}
\end{equation}
    then $\varphi(N)= N_0$, as in \cref{N0}.
\end{thm}

The structure of our proof of \cref{thm:reconstruct} parallels the proof of \cite[Theorem~6.6]{ACCCLMRSS}. First, we construct the groupoid $\Sigma$ of ultrafilters of the Cartan inverse semigroup $N$, and the groupoid $G$ which will be the quotient of $\Sigma$ with respect to an equivalence relation defined in \cref{quotientaction}. We claim that $G$ is Hausdorff and ample, and $\Sigma \twoheadrightarrow G$ is a discrete $R$-twist.

We first recall the definitions of a filter and of an ultrafilter.
For $U\subseteq N$, the \textit{upclosure} of $U$ is the set
  	\[ U\up:= \{m\in N: \text{there exists $n\in U$ with $n\leq m$}\}.\]
 A \textit{filter} of $N$ is a subset $U \subseteq N\backslash \{0\}$ such that $U=U\up$, and whenever $m,n \in U$, there exists $l\in U$ such that $l\leq m,n$. An \textit{ultrafilter} is a maximal filter with respect to $\subseteq$. 

 Let $\Sigma$ be the set of all ultrafilters of $N$. With the following properties, $\Sigma$ is a groupoid (see \cite[Proposition~9.2.1]{Lawson} and \cite[Proposition 2.13]{Lawson2}).
 \begin{enumerate}[label=(\arabic*), font=\normalfont]
      \item For each $U\in \Sigma$, $U^{-1}:= \{ n^{\dagger} : n \in U\}$.
      \item A pair $(U,V)$ of ultrafilters of $N$ is composable if and only if $m^{\dagger}mnn^{\dagger} \neq 0$ (or, equivalently, $mn\neq 0$) for all $m \in U$ and $n\in V$, and  $UV:= \{mn:m \in U, n\in V \}\up.$ Equivalently, a pair $(U,V)$ of ultrafilters of $N$ is composable if and only if $\source(U):=U^{-1}U$ is equal to $\range(V):= VV^{-1}$.    
      \item The unit space of $\Sigma$ is given by $\Sigma^{(0)}= \{U \in \Sigma : U \cap I(B_N) \neq \emptyset\}$.
  \end{enumerate}
  
 For $n\in N$, we denote \[\mathcal{V}_n:= \{ U \in \Sigma: n \in U \}.\] The collection $\{\mathcal{V}_n: n \in N \}$ forms a basis of compact open bisections for a topology on $\Sigma$, making it an \'{e}tale groupoid; moreover, it is ample by \cite[Lemma 5.5]{ACCCLMRSS}. For any $n,m \in N$, we have $\mathcal{V}_n\mathcal{V}_m=\mathcal{V}_{nm}$ and $\mathcal{V}_n^{-1}=\mathcal{V}_{n^{\dagger}}$ (see \cite[Lemma~3.2]{germs} for more details). By \cite[Lemma~2.22]{Lawson3} (see also \cite[Propositions~2.2 and 4.4]{germs}), $\Sigma^{(0)}$ is a Hausdorff subspace of $\Sigma$. Moreover, one can check that various properties about $\Sigma$ in \cite[Section~5.1]{ACCCLMRSS} hold true in this setting; in particular, for $t\in R^{\times}$ and $U \in \Sigma$, we have $tU=\{tm: m\in U \} \in \Sigma$ and $\source(tU)= \source(U)$.

Following the construction in \cite{ACCCLMRSS}, the condition below defines an equivalence relation on $\Sigma$ 
 	\begin{equation} \label{quotientaction}
 	    U\simeq W \iff \text{there exists $t\in R^{\times}$ such that $U=tW$.}
	\end{equation}
Let $G$ be the quotient $\Sigma/{ \simeq}$ of $\Sigma$ with respect to the equivalence relation $\simeq$ defined in \eqref{quotientaction}. We denote the corresponding quotient map by $q:\Sigma \to G$ and the equivalence class of $U\in \Sigma$ by $q(U)$. Applying the techniques used in \cite[Section~5.2]{ACCCLMRSS}, one can see that $G:= \{q(U): U \in \Sigma\}$ is a groupoid with inversion given by $q(U)^{-1}=q(U^{-1})$, composable pairs $G^{(2)}:=\{ (q(U), q(W)): (U,W) \in \Sigma^{(2)}\}$ and the composition $q(U)q(W)=q(UW)$ for all $(U,W) \in \Sigma^{(2)}$. Furthermore, for any $U\in \Sigma$, we have $\source(q(U))=q(\source(U))$ and $\range(q(U))= q(\range(U))$; thus, $G^{(0)}=q(\Sigma^{(0)})$. The collection $\{q(\mathcal{V}_n): n\in N\}$ forms a basis of compact open bisections for the quotient topology on $G$; in particular, $G$ is an ample groupoid. 

Define $i: G^{(0)} \times R^{\times} \to \Sigma$ by $i(q(U),t)= tU$ for $U\in \Sigma^{(0)}$ and $t\in R^{\times}$. Suppose that $B_N\subseteq A$ is a semi-Cartan subalgebra associated to the Cartan inverse semigroup $N$. For any $n\in N$, by \cref{equistat:ibi}, there exists $f\in I(N)= I(B_N)$ such that $P(n)= nf= fn\in N\cap B_N$; thus, following the proof of \cite[Theorem 5.10]{ACCCLMRSS}, $G$ is Hausdorff. Moreover, following the proof of \cite[Theorem 5.6]{ACCCLMRSS}, the sequence 
     \[
 		G^{(0)} \times R^{\times} \overset{i} \hookrightarrow \Sigma \overset{q} \twoheadrightarrow G 
 	\] 
     is a discrete $R$-twist.

To prove \cref{thm:reconstruct}, it remains to show that there is an isomorphism from the $R$-algebra $A$ onto the twisted Steinberg algebra $A_R(G; \Sigma)$ which maps $B_N$ onto $A_R(G^{(0)})$; furthermore, if $N$ is summable then $N$ is mapped onto $N_0$.

Note that when $R$ is indecomposable, the idempotents of $A_R(G^{(0)})$ are characteristic functions of compact open subsets of $G^{(0)}$. Moreover, since $B_N$ satisfies \pref{WT} by \cref{res: scB}, the map $\Sigma^{(0)} \supseteq U \mapsto \tilde{1}_U\in A_R(G;\Sigma)$, where $\spp(\tilde{1}_U)\subseteq R^\times \cdot U$ and ${\tilde{1}_U|}_U= 1_U$ (here $1_U: \Sigma\to R$ is the characteristic function of the compact open subset $U$ of $\Sigma$), is a lattice isomorphism between compact open subsets of $\Sigma^{(0)}$ and $I(B_N)$. By \cite[Th\'{e}or\`{e}me~1]{isophi}, there is an isomorphism $\phi: B_N\to A_R(\Sigma^{(0)})$ satisfying 
 \begin{equation}\label{eq:phi on B}
 \phi(e)= 1_{\mathcal{V}_e} \quad \text{for all} \quad e\in I(B_N)= I(N).
 \end{equation}

\begin{prop}\label{res: varphi-injective}
  Under the hypotheses of \cref{thm:reconstruct},
    for each $a\in A$, define a function $\widehat{a}:\Sigma \to R$ such that for any ultrafilter $U\in \Sigma$ and any $n\in U$, 
    \begin{equation}\label{hata}
    \widehat{a}(U)= \phi(P(n^\dagger a))(\source(U)).
    \end{equation}
    Then the map $a\mapsto \widehat{a}$ from $A$ to $C(\Sigma, R)$ is $R$-linear and injective; moreover, $\widehat{a}$ is $R^{\times}$-contravariant. 
\end{prop}

\begin{proof}
To see that $\widehat{a}$ is a well-defined function satisfying the given formula, we must show that $\phi(P(n^\dagger a))(\source(U))$ is independent of the choice of $n\in U$. Fix $a\in A$, $U\in \Sigma$ and $m, n\in U$. We claim that 
\[
\phi(P(n^\dagger a))(\source(U))- \phi(P(m^\dagger a))(\source(U))= 0.
\]
As $m, n\in U$, there exists $k\in U$ such that $k\leq m, n$, whence $k^\dagger\leq m^\dagger, n^\dagger$. So $k^\dagger k m^\dagger= k^\dagger= k^\dagger k n^\dagger$; moreover, $\phi(k^\dagger k)(\source(U))= 1_{\mathcal{V}_{k^\dagger k}}(U^{-1}U)= 1$. Now using that $\phi$ is multiplicative and that $P$ is a conditional expectation, we have $\phi(P(n^\dagger a))(\source(U))- \phi(P(m^\dagger a))(\source(U))= 0$. 

We show that $\widehat{a}$ is locally constant; for this, the proof of \cite[Proposition~6.1(a)]{ACCCLMRSS} applies.

Since both $P$ and $\phi$ are $R$-linear, the map $a\mapsto \widehat{a}$ is $R$-linear. To show that $a\mapsto \widehat{a}$ is injective: suppose that $\widehat{a}= 0$, then $\widehat{a}(U)= 0$ for all $U\in \Sigma$. It is sufficient to show that $\phi(P(na))= 0$ for all $n\in N$; since $P$ is faithful by \cref{Pfaithful}, we can deduce that $a= 0$. Fix $n\in N$ and $U\in \Sigma^{(0)}$. First, suppose that $nn^\dagger \in U$, then 
\[
U\in \mathcal{V}_{nn^\dagger}= \mathcal{V}_n \mathcal{V}_{n^\dagger}= \source(\mathcal{V}_{n^\dagger}),
\]
and thus we can find an ultrafilter $W\in \mathcal{V}_{n^\dagger}$ such that $\source(W)= U$. We then have
\[
\phi(P(na))(U)= \phi(P(na))(\source(W))= \widehat{a}(W)= 0.
\]
Suppose that $nn^\dagger \notin U$, then $U\notin \mathcal{V}_{nn^\dagger}$. Since $P$ is a conditional expectation, we have
\[
\phi(P(na))(U)= \phi(P(nn^\dagger na))(U)= \phi(nn^\dagger)(U)\phi(P(na))(U)= 1_{\mathcal{V}_{nn^\dagger}}(U)\phi(P(na))(U)= 0,
\]
as needed. Finally, fix $U\in \Sigma$, $t\in R^\times$ and let $n\in U$. Then $tn\in tU$ and $\source(tU)= \source(U)$. Thus, 
\[
\widehat{a}(tU)= \phi(P((tn)^\dagger a))(\source(tU))= \phi(P(t^{-1}n^\dagger a))(\source(U))= t^{-1}\phi(P(n^\dagger a))(\source(U))= t^{-1}\widehat{a}(U),
\]
that is, $\widehat{a}$ is $R^\times$-contravariant. 
\end{proof}

The idempotents inside the inverse semigroup of compact open bisections of $\Sigma$ are the $\mathcal{V}_e$, $e\in I(B_N)$, and are closed under relative complement and disjoint union, i.e., $\mathcal{V}_e \setminus \mathcal{V}_f= \mathcal{V}_{e- ef}$ for any $e, f\in I(B_N)$ and $\mathcal{V}_e \cap \mathcal{V}_f= \emptyset$ if $ef= 0$. As a consequence, for any compact open bisection $V$, by \cite[Lemma 6.4]{ACCCLMRSS} we have $V= \bigsqcup_{i=1}^{k} \mathcal{V}_{n_i}$, where $\mathcal{V}_{n_i}\subseteq V$, for some $k\in \mathbb N$ and $n_i\in N$. 
Consequently, we have a similar proposition as \cite[Proposition 6.5]{ACCCLMRSS}.

\begin{prop}\label{prop:decom1}
    For any $f\in A_R(G;\Sigma)$, there exist $n_1, \ldots, n_k\in N$ and $t_1, \ldots, t_k\in R$ such that $f= \sum_{i=1}^{k} t_i \widehat{n}_i$. In particular, for any compact open bisection $V$ of $\Sigma$, we have $\tilde{1}_V= \sum_{i=1}^k \widehat{n}_i$ for some $n_1, \ldots, n_k\in N$.
\end{prop}

\begin{proof}[Proof of \cref{thm:reconstruct}] 
Define $\varphi: A \to C(\Sigma, R)$ by $ a\mapsto \widehat{a}$. We begin by showing that $\widehat{a} \in A_R(G; \Sigma)$ for every $a\in A$; however, as $\varphi$ is $R$-linear and $N$ spans $A$, it suffices to show that $\widehat{n}\in A_R(G;\Sigma)$ for all $n\in N$. Take $U\in \Sigma$, if $U\in \mathcal{V}_{tn}$ for some $t\in R^\times$, then
\[
\widehat{n}(U)= \phi(P(t^{-1}n^\dagger n))(\source(U))= \phi(t^{-1}n^\dagger n)(\source(U))= t^{-1}\phi(n^\dagger n)(\source(U)).
\]
Since $\phi: B_N \to A_R(\Sigma^{(0)})$ is an isomorphism satisfying \eqref{eq:phi on B}, $\widehat{n}(U)= t^{-1}1_{\mathcal{V}_{n^\dagger n}}(\source(U))= t^{-1}$ as $n^\dagger n\in \source(U)$. Otherwise, $\widehat{n}(U)= 0$, implying that $\widehat{n}= \tilde{1}_{\mathcal{V}_n}\in A_R(G;\Sigma)$, as needed.

\cref{res: varphi-injective} shows that $\varphi$ is an injective $R$-linear map.  \cref{prop:decom1} says that $\tilde{1}_{V}$ is in the image of $\varphi$ for every compact open bisection $V$ of $\Sigma$ and hence $\varphi$ is surjective.  

To see that $\varphi$ is an isomorphism,  we have to show that $\widehat{a a'}= \widehat{a}\ast \widehat{a'}$ for all $a, a'\in A$; however, since $A$ is the $R$-linear span of $N$, it suffices to show that $\widehat{nm}= \widehat{n}\ast \widehat{m}$ for any $n, m\in N$. But this follows from the fact that $\widehat{n}=\til_{\V_n}$ and that $\mathcal{V}_n \mathcal{V}_m= \mathcal{V}_{nm}$.

To check the statement for $B_N$: if $e\in I(B_N)= I(N)$, then $\varphi(e)= \widehat{e}= \tilde{1}_{\mathcal{V}_e}$. Since $I(B_N)$ spans $B_N$ by \cref{res: scB}, and $\{\tilde{1}_{\mathcal{V}_e}: e\in I(B_N)\}$ spans $A_R(G^{(0)};q^{-1}(G^{(0)}))$ which is isomorphic to $A_R(\Sigma^{(0)})$ and to $A_R(G^{(0)})$, it follows that $\varphi$ maps $B_N$ isomorphically to $A_R(\Sigma^{(0)})$.   
 
Now, suppose that $N$ is summable. To see that $\varphi(N)= N_0$, first note that for any $n\in N$, $\widehat{n}= \tilde{1}_{\mathcal{V}_n}\in N_0$. For the inclusion $N_0\subseteq \varphi(N)$: Fix a compact open bisection $V$. Then $V=\bigsqcup_{i=1}^{k} \mathcal{V}_{n_i}$, where $\mathcal{V}_{n_i}\subseteq V$, for some $k\in \mathbb N$ and $n_i\in N$. By \cref{prop:decom1}, we have $\tilde{1}_V= \sum_{i=1}^{k} \tilde{1}_{\mathcal{V}_{n_i}}$.
As $\mathcal{V}_{n_i^\dagger n_j}= \mathcal{V}^{-1}_{n_i} \mathcal{V}_{n_j}\subseteq V^{-1}V\subseteq \Sigma^{(0)}$, thus, $\varphi(n_i^\dagger n_j)\in \varphi(B_N)$. Since $\varphi$ is injective, we have $n_i^\dagger n_j\in B_N$. Similarly, $n_i n_j^\dagger\in B_N$. Since $N$ is summable, we have $n_i+ n_j\in N$ for all $1\leq i, j\leq k$; consequently, $\sum_{i=1}^k n_i\in N$. Hence, $\tilde{1}_V= \sum_{i=1}^k \tilde{1}_{\mathcal{V}_{n_i}}= \varphi(\sum_{i=1}^k n_i)\in \varphi(N)$, completing the proof.
\end{proof}

\section{Recovering the twist}\label{sec:recovery}

Let $\Sigma \twoheadrightarrow G$ be a discrete $R$-twist over an ample Hausdorff groupoid $G$ where $R$ is an indecomposable commutative unital ring. In the next result, we show that the groupoid $\Sigma_{N_0}$ associated to the Cartan inverse semigroup $N_0= \{\widetilde{1}_X : X\in\cob\}\subseteq A_R(G;\Sigma)$, where $\cob$ denotes the collection of compact open bisections of $\Sigma$, recovers $\Sigma$.

\begin{thm}\label{thm: recovering}
Let $q\colon\Sigma \twoheadrightarrow G$ be a discrete $R$-twist over an ample Hausdorff groupoid $G$ where $R$ is an indecomposable commutative unital ring. Let $A= A_R(G;\Sigma)$, $B=A_R(G^{(0)})$, and consider the Cartan inverse semigroup $N_0\subseteq A$. Then for each $\sigma\in \Sigma$, the set 
\[
U_\sigma:= \{\widetilde{1}_X\in N_0 : X\in\cob,\, \sigma\in X\}
\]
is an ultrafilter of $N_0$, and the map $\Phi\colon \Sigma\to \Sigma_{N_0}$ given by $\Phi(\sigma)=U_\sigma$ is an $R^\times$-equivariant isomorphism of topological groupoids.
\end{thm}

\begin{proof}
We start by showing that $U_\sigma$ is an ultrafilter of $N_0$. We use \cite[Lemma~8.2(c)]{ACCCLMRSS} which says that for $X,Y\in\cob$ we have $\widetilde{1}_X\le \widetilde{1}_Y\iff X\subseteq Y$. So if $\widetilde{1}_X\in U_\sigma$ and $\widetilde{1}_Y\in N_0$ with $\widetilde{1}_X\le \widetilde{1}_Y$, then $X\subseteq Y$. Hence $\sigma\in Y$, and so $\widetilde{1}_Y\in U_\sigma$, and we see that $U_\sigma$ is upwards closed. If $\widetilde{1}_X,\widetilde{1}_Y\in U_\sigma$, then $\sigma \in X\cap Y$, and so $\widetilde{1}_{X\cap Y}\in U_\sigma$ with $\widetilde{1}_{X\cap Y}\le \widetilde{1}_X$ and $\widetilde{1}_{X\cap Y}\le \widetilde{1}_Y$. Hence $U_\sigma$ is downwards directed. To see that $U_\sigma$ is maximal, suppose towards a contradiction that $\mathcal{F}\subseteq N_0\setminus\{0\}$ is a filter that properly contains $U_\sigma$. Then there is $Y\in\cob$ such that $\widetilde{1}_Y\in\mathcal{F}\setminus U_\sigma$, and so $\sigma\notin Y$. Since $\Sigma$ is Hausdorff, we can choose $X\in\cob$ with $\sigma \in X$ and $X\cap Y=\emptyset$. Then $\widetilde{1}_X\in U_\sigma\subseteq \mathcal{F}$, and hence there must be some $W\in\cob$ with $\widetilde{1}_W\le \widetilde{1}_X,\widetilde{1}_Y$. But this forces $W\subseteq X\cap Y=\emptyset$, which means $0=\widetilde{1}_W\in \mathcal{F}$. This is a contradiction, and hence we must have $\mathcal{F}=U_\sigma$. We see that $U_\sigma$ is maximal, hence is an ultrafilter.

We define $\Phi\colon \Sigma\to \Sigma_{N_0}$ by $\Phi(\sigma)=U_\sigma$. We first claim that for each $\sigma\in\Sigma$, we have 
\[
\range_{N_0}(\Phi(\sigma))= \Phi(\range(\sigma))\quad\text{and}\quad \source_{N_0}(\Phi(\sigma))= \Phi(\source(\sigma)),
\]
where $\range_{N_0}$ and $\source_{N_0}$ are the range and source maps in $\Sigma_{N_0}$, respectively.
First note that 
\[
\range_{N_0}(\Phi(\sigma))=\{\widetilde{1}_X\widetilde{1}_Y^\dagger : X,Y\in\cob,\,\sigma\in X\cap Y\}\quad\text{and}\quad \Phi(\range(\sigma))=\{\widetilde{1}_W : W\in\cob,\, \range(\sigma)\in W\}.
\]
If $\sigma\in X\cap Y$, then $\range(\sigma)=\sigma\sigma^{-1}\in XY^{-1}$, and so $\widetilde{1}_X\widetilde{1}_Y^\dagger=\widetilde{1}_{XY^{-1}}\in \Phi(\range(\sigma))$, and hence $\range_{N_0}(\Phi(\sigma))\subseteq \Phi(\range(\sigma))$. Since both sets are ultrafilters, we get equality. A similar argument gives $\source_{N_0}(\Phi(\sigma))=\Phi(\source(\sigma))$, and the claim holds.

To see that $\Phi$ is multiplicative, suppose $\sigma,\tau\in\Sigma$ are composable. We claim that $\Phi(\sigma\tau)=\Phi(\sigma)\Phi(\tau)$. We have
\[
\Phi(\sigma)\Phi(\tau)=\{\widetilde{1}_X\widetilde{1}_Y: X,Y\in\cob,\, \sigma\in X,\,\tau\in Y\}.
\]
If $X,Y\in\cob$ with $\sigma\in X,\,\tau\in Y$, then  $\sigma\tau\in XY$ and hence $\widetilde{1}_X\widetilde{1}_Y=\widetilde{1}_{XY}\in \Phi(\sigma\tau)$. Hence $\Phi(\sigma)\Phi(\tau)\subseteq \Phi(\sigma\tau)$, and since both are ultrafilters, we have equality. So $\Phi$ is multiplicative. 

To see that $\Phi$ is injective, let $\sigma,\tau\in\Sigma$ with $\sigma\not=\tau$. Then for $X\in\cob$ with $\sigma\in X$ and $\tau\notin X$ we have $\widetilde{1}_X\in \Phi(\sigma)$ and $\widetilde{1}_X\notin \Phi(\tau)$. Hence $\Phi(\sigma)\not=\Phi(\tau)$. 

For the $R^\times$-equivariance, first note that for $t\in R^\times$ and $X\in \cob$ we have $t\widetilde{1}_X=\widetilde{1}_{t\cdot X}$. So for $\sigma\in\Sigma$ we have
\[
t\cdot \Phi(\sigma)=\{\widetilde{1}_{t\cdot X} : X\in\cob,\,\sigma\in X\}\quad\text{and}\quad \Phi(t\cdot\sigma)=\{\widetilde{1}_X : X\in\cob,\, t\cdot\sigma\in X\}.
\]
Then $\sigma\in X\implies t\cdot\sigma\in t\cdot X\implies \widetilde{1}_{t\cdot X}\in\Phi(t\cdot\sigma)$. So $t\cdot\Phi(\sigma)\subseteq \Phi(t\cdot \sigma)$, and since they are both ultrafilters, we get equality.

We show that $\Phi$ is surjective. Let $U\in\Sigma_{N_{0}}$ and set $\mathcal{X}:=\{X\in\Sigma_{\mathrm{cob}} : \widetilde{1}_{X}\in U\}$. Since $ U\subseteq N_{0}\setminus\{0\}$ and $\widetilde{1}_{\emptyset}=0$, each $X\in\mathcal{X}$ is a nonempty compact subset of $\Sigma$. Moreover $\mathcal{X}$ is downward directed: if $\widetilde{1}_{X},\widetilde{1}_{Y}\in  U$ then, as $ U$ is a filter, there is
$\widetilde{1}_{W}\in U$ with $\widetilde{1}_{W}\leq\widetilde{1}_{X}, \widetilde{1}_{Y}$, and hence $W\subseteq X\cap Y$.

Fix $X_{0}\in\mathcal{X}$. Since $\Sigma$ is Hausdorff, every $X\in\mathcal{X}$ is closed, so $\{X\cap X_{0} : X\in\mathcal{X}\}$ is a family of closed subsets of the compact set $X_{0}$, and it has the finite intersection property because $\mathcal{X}$ is downward directed. Hence $\bigcap\mathcal{X}\neq\emptyset$. Choose $\sigma\in\bigcap\mathcal{X}$, and let $X\in\Sigma_{\mathrm{cob}}$ with $\widetilde{1}_{X}\in U$. Then $X\in\mathcal{X}$, so $\sigma\in X$ and hence $\widetilde{1}_{X}\in U_{\sigma}$. Since every element of $U$ is of the form $\widetilde{1}_{X}$ for some $X\in\Sigma_{\mathrm{cob}}$, this shows $U\subseteq U_{\sigma}$. As $U$ is an ultrafilter, we conclude $U=U_{\sigma}=\Phi(\sigma)$. So $\Phi$ is surjective.

Finally, for continuity and openness of $\Phi$ we use the basis elements $\mathcal{V}_X:=\{U\in\Sigma_{N_0} : \widetilde{1}_X\in U\}$ for $X\in \cob$. Since
\[
\Phi^{-1}(\mathcal{V}_X)=\{\sigma\in\Sigma : \widetilde{1}_X\in U_\sigma\}=\{\sigma\in \Sigma : \sigma \in X\}=X,
\]
it follows that $\Phi$ is continuous and open. 
\end{proof}

\vspace{.5cm}

\printbibliography

\end{document}

\typeout{get arXiv to do 4 passes: Label(s) may have changed. Rerun}